\documentclass[12pt, reqno]{amsart}

\usepackage{amsmath,amsfonts,amssymb,amsthm,enumerate,hyperref,multicol}
\newtheorem{lemma}{Lemma}
\newtheorem{theorem}{Theorem}

\newtheorem{remark}{Remark}
\newtheorem{Proposition}{Proposition}
\begin{document}
\leftline{ \scriptsize \it }

\title[Moment Estimates]{Moment Estimations of new \\Sz\'asz-Mirakyan-Durrmeyer operators}
\maketitle

\begin{center}
{\bf Vijay Gupta} \\
Department of Mathematics, \\
Netaji Subhas Institute of Technology \\
Sector 3 Dwarka, New Delhi-110078, India \\
vijaygupta2001@hotmail.com \\
\vskip0.2in
{\bf G. C. Greubel} \\
Newport News, VA, United States \\
jthomae@gmail.com
\end{center}

\begin{abstract}
In \cite{J} Jain introduced the modified form of the Sz\'asz-Mirakjan operator, based on certain parameter $0 \leq 
\beta < 1$. Several modifications of the operators proposed and are available in the literature. Here we consider 
actual Durrmeyer variants of the operators due to Jain. It is observed here that the Durrmeyer variant have nice 
properties and one need not to take any restriction on $\beta$ in order to obtain convergence. We establish
moments using the Tricomi's confluent hypergeometric function and Stirling numbers of first kind, and also
estimate some direct results.
\end{abstract}

\vspace{3mm}

\section{Introduction}
The well known Sz\'{a}sz-Mirakyan operators are defined as
\begin{align*}
B_{n}^{0}(f,x) = \sum_{k=0}^\infty e^{-nx}\frac{(nx)^k}{k!} f(k/n), \hspace{10mm} x \in [0,\infty).
\end{align*}
In order to generalize the Sz\'asz-Mirakyan operators Jain, \cite{J}, introduced the following operators
\begin{align}\label{e1}
B_{n}^{\beta}(f,x) = \sum_{k=0}^\infty L_{n,k}^{(\beta)}(x) f(k/n), \hspace{10mm} x \in [0,\infty)
\end{align}
where $0 \leq \beta<1$ and the basis function is defined as
\begin{align*}
L_{n,k}^{(\beta)}(x) = \frac{nx(nx+k\beta)^{k-1}}{k!}e^{-(nx+k\beta)}
\end{align*}
where it is seen that $\sum_{k=0}^{\infty}L_{n,k}^{(\beta)}(x)=1$. As a special case when $\beta=0$,
the operators (\ref{e1}) reduce to the Sz\'{a}sz-Mirakyan operators. Umar and Razi, \cite{UR}, used
a Kantorovich type modification of $L_{n,k}^{(\beta)}(x)$ in order to approximate integrable functions,
where some direct estimates were considered. Recently Farca\c{s}, \cite{af}, studied the operators 
(\ref{e1}) and estimated a Voronovskaja type asymptotic formula. While a review of Farca\c{s}' work 
was undertaken it was found that minor errors were given in Lemma 2.1 . These errors have been corrected 
and are given in Lemma 1.

In 1967 Durrmeyer, \cite{Dur}, introduced the integral modification of the well known Bernstein polynomials, 
which were later studied by Derriennic \cite{der}, Gonska-Zhou \cite{gon1},\cite{gon2} and Agrawal-Gupta \cite{pna}.  
The Durrmeyer type modification of the operators (\ref{e1}), with different weight functions, have been
proposed by Tarabie \cite{st} and Gupta et al \cite{vg1}. In approximations by linear positive operators, moment
estimations play an important role. So far no standard Durrmeyer type modification of the operators (\ref{e1}) have
been discussed due to its complicated form in finding moments and this problem has not been discussed in the last 
four decades. Here we overcome this difficulty and we consider the following Durrmeyer variant of the operators
(\ref{e1}) in the form
\begin{align}\label{e2}
D_{n}^{\beta}(f,x) &= \sum_{k=0}^{\infty}\left(\int_0^\infty L_{n,k}^{(\beta)}(t)\,dt \right)^{-1}L_{n,k}^{(\beta)}
(x)\int_{0}^{\infty}L_{n,k}^{(\beta)}(t)f(t)\,dt \nonumber\\
&= \sum_{k=0}^{\infty} \frac{< L_{n,k}^{(\beta)}(t), f(t) >}{< L_{n,k}^{(\beta)}(t), 1 >} \, L_{n,k}^{(\beta)}(x)
\end{align}
where $< f,g> = \int_{0}^{\infty} f(t)g(t)dt$. For the special case of $\beta = 0$ these operators reduce to the 
Sz\'asz-Mirakyan-Durrmeyer operators (see \cite{vgrp} and references therein). It has been observed that 
these operators have interesting convergence properties. In the original form of the operators (\ref{e1}) 
and its other integral modifications, one has to consider the restriction that $\beta \to 0$ as $n \to \infty$, 
in order to obtain convergence. For these actual Durrmeyer variants, (\ref{e2}), we need not to take any 
restrictions on $\beta$. Because of this beautiful property it is of worth to study these operators. Here
we find moments using Stirling numbers of first kind and confluent hypergeometric function and estimate some
basic direct results.

\section{Moments}

\begin{lemma} \cite{J}, \cite{af} \label{l1}
For the operators defined by (\ref{e1}) the moments are as follows:
\begin{align}
B_{n}^{\beta}(1,x) &= 1, \hspace{10mm} B_n^{\beta}(t,x)=\frac{x}{1-\beta} \nonumber\\
B_{n}^{\beta}(t^{2},x) &= \frac{x^2}{(1-\beta)^2}+\frac{x}{n(1-\beta)^3}, \nonumber\\
B_{n}^{\beta}(t^{3},x) &= \frac{x^3}{(1-\beta)^3}+\frac{3 \, x^2}{n(1-\beta)^4}+\frac{(1+2\beta)\, x}
{n^2(1-\beta)^5} \label{e3} \\ 
B_{n}^{\beta}(t^{4},x) &= \frac{x^4}{(1-\beta)^4}+\frac{6 \, x^3}{n(1-\beta)^5}+\frac{(7+8\beta) x^2}
{n^2(1-\beta)^6}  +\frac{(6\beta^2 + 8\beta +1) x}{n^3(1-\beta)^7} \nonumber\\
B_{n}^{\beta}(t^{5}, x) &= \frac{x^{5}}{(1-\beta)^{5}} + \frac{10 \, x^{4}}{n(1-\beta)^{6}} + \frac{
5(4 \beta + 5) \, x^{3}}{n^{2}(1-\beta)^{7}} \nonumber\\
& \hspace{10mm} + \frac{15(2 \beta^{2} + 4 \beta + 1) \, x^{2}}{ n^{3}
(1-\beta)^{8}} + \frac{(24\beta^{3} + 58 \beta^{2} + 22 \beta + 1) \, x}{n^{4}(1-\beta)^{9}} \nonumber
\end{align}
\end{lemma}

\begin{lemma} \label{l2} For $0\le \beta<1$, we have
\begin{align}
\frac{< L_{n,k}^{(\beta)}(t), t^{r} >}{< L_{n,k}^{(\beta)}(t), 1 >} = P_{r}(k; \beta) \label{e4}
\end{align}
where $< f,g> =\int_{0}^{\infty} f(t)g(t)dt$ and $P_{r}(k; \beta)$ is a polynomial of order $r$ in the variable $k$.
In particular
\begin{align}
P_{0}(k; \beta) &= 1 \nonumber\\
P_{1}(k; \beta) &= \frac{1}{n} \left[ (1-\beta) k + \frac{1}{1-\beta} \right], \nonumber\\
P_{2}(k; \beta) &= \frac{1}{n^{2}} \left[ (1-\beta)^{2} k^{2} + 3 k + \frac{2!}{1-\beta} \right], \label{e5} \\
P_{3}(k; \beta) &= \frac{1}{n^{3}} \left[ (1-\beta)^{3} k^{3} + 6(1-\beta) k^{2} + 
\frac{(11-8\beta) \, k}{1-\beta} + \frac{3!}{1-\beta} \right], \nonumber\\
P_{4}(k; \beta) &= \frac{1}{n^{4}} \left[ (1-\beta)^{4} k^{4} + 10 (1-\beta)^{2} k^{3} + 5(7-4\beta) k^{2} 
+ \frac{10(5-3\beta) \, k}{1-\beta} + \frac{4!}{1-\beta} \right] \nonumber\\
P_{5}(k;\beta) &= \frac{1}{n^{5}} \left[ (1-\beta)^{5} k^{5} + 15 (1-\beta)^{3} k^{4} + 5 (1-\beta)(17 - 8 \beta) k^{3}
\right. \nonumber\\
& \hspace{20mm} \left. + \frac{15(15-20\beta + 6\beta^{2}) \, k^{2}}{1-\beta} + \frac{(274-144\beta) \, k}{1-\beta}
+ \frac{5!}{1-\beta} \right] \nonumber
\end{align}
\end{lemma}

\begin{proof} First, we consider the integral:
\begin{align*}
<L_{n,k}^{(\beta)}(t), t^{r} > &= \int_{0}^{\infty} L_{n,k}^{(\beta)}(t) \, t^{r} \, dt  \\
&= \frac{n}{k!}\int_0^\infty e^{-(nt+k\beta)}t^{r+1}(nt+k\beta)^{k-1}dt
\end{align*}
We use Tricomi's confluent hypergeometric function:
\begin{align*}
U(a,b,c) = \frac{1}{\Gamma(a)}\int_0^\infty e^{-ct}t^{a-1}(1+t)^{b-a-1},a>0,c>0
\end{align*}
we have
\begin{align}
<L_{n,k}^{(\beta)}(t), t^{r} > &= \frac{n}{k!}\int_0^\infty e^{-(nt+k\beta)}t^{r+1}(nt+k\beta)^{k-1}dt \nonumber\\
&= \frac{1}{k!}\int_0^\infty (x+k\beta)^{k-1}e^{-(x+k\beta)}\left(\frac{x}{n}\right)^{r+1}dx \nonumber\\
&= \frac{(k\beta)^{k+r+1}}{k!n^{r+1}}e^{-k\beta}\int_0^\infty e^{-k\beta t}(1+t)^{k-1}t^{r+1}dt \nonumber\\
&= \frac{(k\beta)^{k+r+1}}{k!n^{r+1}}e^{-k\beta}(r+1)!U(r+2,k+r+2,k\beta). \label{e6}
\end{align}
The evaluation of $<L_{n,k}^{(\beta)}(t), t^{r} >$ can also be seen in the form
\begin{align*}
<L_{n,k}^{(\beta)}(t), t^{r} > &= \frac{(r+1)!}{k \, n^{r+1}} \, e^{-k \beta} \, \sum_{s=0}^{k-1} \binom{k+r-s}{r+1}
\frac{(k\beta)^{s}}{s!} \nonumber\\
&= \frac{e^{-x}}{k n^{r+1}} \, \sum_{s=0}^{k-1} \phi_{r}(s) \frac{x^{s}}{s!}
\end{align*}
where $x = \beta k$ and $\phi_{r}(s)$ is given by
\begin{align}
\phi_{r}(s) &= (k-s)_{r+1} = \sum_{j=0}^{r+1} s(r+1, r-j+1) (k-s)^{r-j+1} \label{e7}
\end{align}
where $s(n,k)$ are the Stirling numbers of the first kind. The first few may be written as
\begin{align*}
\phi_{0} &= k-s \\
\phi_{1} &= (k-s)^{2} + (k-s) \\
\phi_{2} &= (k-s)^{3} + 3 (k-s)^{2} + 2(k-s) \\
\phi_{3} &= (k-s)^{4} + 6 (k-s)^{3} + 11 (k-s)^{2} + 6 (k-s)
\end{align*}
It can now be determined that
\begin{align}
<L_{n,k}^{(\beta)}(t), t^{r} > = \frac{e^{-x}}{k n^{r+1}} \, \sum_{j=0}^{r+1} s(r+1, r-j+1) \theta_{r-j}(x) \label{e8}
\end{align}
where
\begin{align*}
\theta_{m}(x) = \sum_{s=0}^{k-1} (k-s)^{m+1} \frac{x^{s}}{s!}.
\end{align*}
For the case of $r=0$, (\ref{e8}) becomes
\begin{align*}
< L_{n,k}^{(\beta)}(t), 1 > = \frac{e^{-x}}{k n} \, \theta_{0}(x) = \frac{e^{-x}}{k n} \, \sum_{s=0}^{k-1} (k-s)
\frac{x^{s}}{s!}
\end{align*}
and for the case $r=1$,
\begin{align*}
< L_{n,k}^{(\beta)}(t), t > = \frac{e^{-x}}{k n^{2}} \, \theta_{1}(x) + \frac{1}{n} < L_{n,k}^{(\beta)}(t), 1 >.
\end{align*}
Dividing both sides by $< L_{n,k}^{(\beta)}(t), 1 >$ leads to the expression
\begin{align}
P_{1}(k; \beta) = \frac{1}{n} \left( 1 + \frac{ \theta_{1}(x) }{ \theta_{0}(x) } \right) = \frac{1}{n}
( 1 + S_{1}(x) ) \nonumber 
\end{align}
where $S_{r}(x)$ is defined by
\begin{align}
S_{r}(x) = \frac{\theta_{r}(x)}{\theta_{0}(x)} = \frac{ \sum_{s=0}^{k-1} (k-s)^{r+1} \frac{x^{s}}{s!} }{ \sum_{s=0}^{k-1}
(k-s) \frac{x^{s}}{s!} }. \label{e9}
\end{align}
The general form of $P_{r}(k; \beta)$ is given by
\begin{align}
P_{r}(k; \beta) = \frac{1}{n^{r}} \, \sum_{j=0}^{r+1} s(r+1, j) S_{j-1}(x). \label{e10}
\end{align}
What remains is to obtain calculations for $S_{r}(x)$. From (\ref{e9}) it is seen that $S_{0}(x) = 1$
and
\begin{align}
S_{1}(x) &= k - \left( \frac{k-1}{k} \right) x +  \left(\frac{x}{k}\right)^{2} +
\left(\frac{x}{k}\right)^{3} + \left(\frac{x}{k}\right)^{4} + \cdots
= k - x + \frac{x}{k-x}  \nonumber\\
S_{2}(x) &= x^{2} - (2k-3) x + k^{2} - \frac{x}{k-x}, \label{e11} \\
S_{3}(x) &= k^{3} - 3k -1 - (3k^{2}- 6k + 7) x + 3(k-2)x^{2} - x^{3} + \frac{k(3k+1)}{k-x}, \nonumber\\
S_{4}(x) &= k^{4} - (4k^{3} - 10k^{2} + 10k -15) x + (6k^{2} -20k + 25) x^{2} \nonumber\\
& \hspace{20mm} - 2(2k-5) x^{3} + x^{4} - \frac{(10k+1)x}{k-x}. \nonumber
\end{align}
Since $x= \beta k$ then the first few $S_{r}(\beta k)$ are seen to be
\begin{align}
S_{1}(\beta k) &= (1-\beta) k + \frac{\beta}{1-\beta} \nonumber\\
S_{2}(\beta k) &= (1-\beta)^{2} k^{2} + 3 \beta k - \frac{\beta}{1-\beta} \label{e12} \\
S_{3}(\beta k) &= (1-\beta)^{3} k^{3} + 6 \beta (1-\beta) k^{2} + \frac{\beta(7\beta-4) \, k}{1-\beta} +
\frac{\beta}{1-\beta} \nonumber\\
S_{4}(\beta k) &= (1-\beta)^{4} k^{4} + 10 \beta (1-\beta)^{2} k^{3} + 5\beta (5 \beta -2) k^{2}
+ \frac{5 \beta (1-3\beta) \, k}{1-\beta} - \frac{\beta}{1-\beta} \nonumber
\end{align}
which are polynomials of order $r$ in the variable $k$. Using the resulting expressions of $S_{r}(\beta k)$,
provided in (\ref{e12}), in (\ref{e10}) lead to the $P_{r}(k; \beta)$ polynomials of (\ref{e5}). It is
now sufficient to conclude that
\begin{align*}
\frac{< L_{n,k}^{(\beta)}(t), t^{r} >}{< L_{n,k}^{(\beta)}(t), 1 >} = P_{r}(k; \beta)
\end{align*}
are polynomials of order $r$ in the variable $k$.
\end{proof}

\begin{lemma}\label{l3}
For $0 \leq \beta < 1$, $r \geq 0$, the polynomials $P_{r}(k; \beta)$ satisfy the recurrence relationship 
\begin{align}
n^{2} P_{r+2}(k; \beta) = n [ (1-\beta) k + r + 2 ] P_{r+1}(k; \beta) + (r+2) \beta k P_{r}(k; \beta). \label{e13}
\end{align}
\end{lemma}

\begin{proof}
By utilizing the recurrence relation, 
\begin{align*}
U(a, b; z) = (a+1) z U(a+2, b+2; z) + (z-b) U(a+1, b+1; z),
\end{align*}
for the Tricomi confluent hypergeometric functions, (\ref{e6}) becomes
\begin{align*}
n^{2} <L_{n,k}^{\beta}(t), t^{r+1}> &= n [ (1-\beta)k + r + 1] <L_{n,k}^{\beta}(t), t^{r}> 
+ (r+1) \beta k <L_{n,k}^{\beta}(t), t^{r-1}>.  
\end{align*}
Now dividing by $<L_{n,k}^{(\beta)}(t), 1>$ leads to the desired relationship for the polynomials 
$P_{r}(k; \beta)$ given by (\ref{e13}). 
\end{proof}


\begin{lemma}\label{l4}
If the $r$-th order moment with monomials $e_r(t)=t^r,r=0,1,\cdots$ of the operators (\ref{e2}) be defined as
\begin{align*}
T_{n,r}^{\beta}(x):D_n^\beta(e_r,x) = \sum_{k=0}^{\infty}\left(\int_0^\infty L_{n,k}^{(\beta)}(t)\,dt \right)^{-1}
L_{n,k}^{(\beta)}(x)\int_{0}^{\infty}L_{n,k}^{(\beta)}(t)t^r\,dt
\end{align*}
or
\begin{align}
T_{n,r}^{\beta}(x) = \sum_{k=0}^{\infty} P_{r}(k; \beta) \, L_{n,k}^{(\beta)}(x). \label{e14}
\end{align}
The first few are:
\begin{align}
T_{n,0}^{\beta}(x) &= 1, \hspace{10mm}  T_{n,1}^{\beta}(x)=x+\frac{1}{n(1-\beta)}, \nonumber\\
T_{n,2}^{\beta}(x) &= x^2+\frac{4x}{n(1-\beta)}+\frac{2!}{n^2(1-\beta)}, \nonumber\\
T_{n,3}^{\beta}(x) &= x^3+\frac{9 \, x^2}{n(1-\beta)} + \frac{6(3-\beta) \, x}{n^2(1-\beta)^2}+
\frac{3!}{n^3(1-\beta)}, \label{e15} \\
T_{n,4}^{\beta}(x) &= x^{4} + \frac{16 \, x^{3}}{n(1-\beta)} + \frac{12(6-\beta) \, x^{2}}{n^{2}(1-\beta)^{2}}
+ \frac{12(3 \beta^{2} - 6 \beta + 8) \, x}{n^{3}(1-\beta)^{3}} + \frac{4!}{n^{4}(1-\beta)}. \nonumber\\
T_{n,5}^{\beta}(x) &= x^{5} + \frac{25 \, x^{4}}{n(1-\beta)} + \frac{20(10- \beta) \, x^{3}}
{n^{2}(1-\beta)^{2}}  + \frac{120(\beta^{2} - 2 \beta + 5) \, x^{2}}{n^{3}(1-\beta)^{3}} \nonumber\\ 
& \hspace{15mm} + \frac{120(5 - 6\beta + 6\beta^{2} - \beta^{3}) \, x}{n^{4}(1-\beta)^{4}} + \frac{5!}{n^{5}(1-\beta)} 
\nonumber 
\end{align}
\end{lemma}

\begin{proof}
Obviously by (\ref{e2}), we have  $T_{n,0}^{\beta}(x)=1.$  Next by definition of $T_{n,r}^{\beta}(x)$, we have
\begin{align*}
T_{n,r}^{\beta}(x) = \sum_{k=0}^{\infty}\frac{< L_{n,k}^{(\beta)}(t), t^r >}{< L_{n,k}^{(\beta)}(t), 1 >} \, 
L_{n,k}^{(\beta)}(x)
= \sum_{k=0}^{\infty} P_{r}(k; \beta) \, L_{n,k}^{(\beta)}(x).
\end{align*}
Using Lemma \ref{l1} and Lemma \ref{l2}, we have
\begin{align*}
T_{n,1}^{\beta}(x) &= \sum_{k=0}^{\infty}L_{n,k}^{(\beta)}(x)P_{1}(k; \beta) 
= \sum_{k=0}^{\infty}L_{n,k}^{(\beta)}(x) \, \frac{1}{n} \left[ (1-\beta) k + \frac{1}{1-\beta} \right]\\
&= (1-\beta)B_n^{\beta}(t,x)+\frac{1}{n(1-\beta)}B_n^{\beta}(1,x)\\
&= x+\frac{1}{n(1-\beta)}.
\end{align*}
\begin{align*}
T_{n,2}^{\beta}(x) &= \sum_{k=0}^{\infty}L_{n,k}^{(\beta)}(x)P_{2}(k; \beta) 
= \sum_{k=0}^{\infty}L_{n,k}^{(\beta)}(x) \, \frac{1}{n^{2}} \left[ (1-\beta)^{2} k^{2} + 3 k + \frac{2}{1-\beta} \right]\\
&= (1-\beta)^2 \, B_{n}^{\beta}(t^2,x) + \frac{3}{n} \, B_{n}^{\beta}(t,x) + \frac{2}{n^2(1-\beta)} \\
&= x^{2} + \frac{4x}{n(1-\beta)} + \frac{2}{n^2(1-\beta)}.
\end{align*}
\begin{align*}
T_{n,3}^{\beta}(x) &= \sum_{k=0}^{\infty}L_{n,k}^{(\beta)}(x)P_{3}(k; \beta) \\
&= \sum_{k=0}^{\infty}L_{n,k}^{(\beta)}(x) \, \frac{1}{n^3} \left[(1-\beta)^3 k^3 + 6(1-\beta)k^2 + 
\frac{(11-8\beta) \, k}{1-\beta} + \frac{3!}{1-\beta}\right] \\
&= (1-\beta)^{3} B_{n}^{\beta}(t^3,x) + \frac{6(1-\beta)}{n} B_{n}^{\beta}(t^2,x) \\
& \hspace{20mm} + \frac{(11-8\beta)}{n^2(1-\beta)} B_{n}^{\beta}(t,x) + \frac{3!}{n^3(1-\beta)}
B_{n}^{\beta}(1,x) \\
&= x^3 + \frac{9 \, x^2}{n(1-\beta)} + \frac{6(3-\beta) \, x}{n^2(1-\beta)^2} + \frac{3!}{n^3(1-\beta)}.
\end{align*}
A continuation of this process will provide $T_{n,r}^{\beta}(x)$ for cases of $r \geq 4$.
\end{proof}

\begin{lemma}\label{l5}
For $r \geq 1$ the polynomials $T_{n,r}^{\beta}(x)$ satisfy the relation
\begin{align}
T_{n,r}^{\beta}(x) = \left( x + \frac{2r-1}{n(1-\beta)} \right) \, T_{n,r-1}^{\beta}(x) - \sum_{j=0}^{r-2}
\frac{(-1)^{j} \, A_{j}^{r-2}}{n^{j+2} (1-\beta)^{j+2}} \, T_{n,r-j-2}^{\beta}(x)    \label{e16}
\end{align}
where the first few coefficients $A_{j}^{r}$ are given by
\begin{align*}
A_{0}^{0} &= 1 + 2 \beta \\
A_{0}^{1} &= 4 + 4 \beta  \hspace{10mm} A_{1}^{1} = 2 \beta + 6 \beta^{2} \\
A_{0}^{2} &= 9 + 6 \beta  \hspace{10mm} A_{1}^{2} = 6 \beta + 30 \beta^{2} \hspace{10mm}
A_{2}^{2} = 12 \beta^{2} + 24 \beta^{3} \\
A_{0}^{3} &= 16 + 8 \beta  \hspace{8mm} A_{1}^{3} = 12 \beta + 84 \beta^{2}  \hspace{8mm}
A_{2}^{3} = 60 \beta^{2} + 96 \beta^{3}  \\
& \hspace{20mm} A_{3}^{3} = -12 \beta^{2} + 48 \beta^{3} + 120 \beta^{4}
\end{align*}
\end{lemma}

\begin{proof}
Making use of (\ref{e13}) and (\ref{e14}) leads to the relation
\begin{align*}
n^{2} T_{n,r}^{\beta}(x) - n (r+1) T_{n,r}^{\beta}(x) = \sum_{k=0}^{\infty} k \left[ n (1-\beta) P_{r}(k; \beta)
+ (r+1) \beta P_{r-1}(k; \beta) \right] L_{n,k}^{(\beta)}(x).
\end{align*}
Now making use of (\ref{e5}) the summation can be reformed into the desired relation. This can be verified
by considering $T_{n,r}^{\beta}(x)$ as a linear combination of $T_{n,j}^{\beta}(x)$ for $0 \leq j \leq r-1$. 
\end{proof}


\begin{remark} \label{r1}
If we denote the central moment as $\mu_{n,r}^\beta(x)=D_{n}^{\beta}((t-x)^r,x)$, then
\begin{align}
\mu_{n,1}^\beta(x) &= \frac{1}{n(1-\beta)}, \hspace{10mm}  \mu_{n,2}^\beta(x) = \frac{2x}{n(1-\beta)}
+\frac{2!}{n^2(1-\beta)} \nonumber\\
\mu_{n,3}^\beta(x) &= \frac{12 \, x}{n^2(1-\beta)^2}+\frac{3!}{n^3(1-\beta)},  \label{e17} \\
\mu_{n,4}^\beta(x) &= \frac{12 \, x^2}{n^2(1-\beta)^2} + \frac{12(6 -2 \beta + \beta^2) \, x}{n^3(1-\beta)^2} 
+ \frac{4!}{n^4(1-\beta)}. \nonumber
\end{align}
In general using the similar approach, one can show that:
$$\mu_{n,r}^\beta(x)=O\left(n^{-[(r+1)/2]}\right),$$
where $[\alpha]$ denotes the integral part of $\alpha.$
\end{remark}

\section{Direct Estimates}
In this section, we establish the following direct result:
\begin{Proposition}\label{t1}
Let $f$ be a continuous function on $[0,\infty)$ for $n\to \infty$, the sequence $\{D_{n}^{\beta}(f,x)\}$ converges
uniformly to $f(x)$ in $[a,b]\subset[0,\infty).$
\end{Proposition}
\begin{proof}
For sufficiently large $n$, it is obvious from Lemma \ref{l4} that $D_{n}^{\beta}(e_0,x),\,\,D_{n}^{\beta}(e_1,x)$ and
$D_{n}^{\beta}(e_2,x)$ converges uniformly to $1,\,\,x$ and $x^2$ respectively on every compact subset of $[0,\infty).$
Thus the required result follows from Bohman-Korovkin theorem.
\end{proof}

\begin{theorem}\label{t2} Let $f$ be a bounded integrable function on $[0,\infty)$ and has second derivative at a point
$x\in [0,\infty)$, then $$\lim_{n\to \infty} n[D_{n}^{\beta}(f,x)-f(x)]=\frac{1}{1-\beta}f^\prime(x)+\frac{x}{1-\beta}
f^{\prime\prime}(x).$$
\end{theorem}
\begin{proof}
By the Taylor's expansion of $f$, we have
\begin{equation}\label{e18}
f(t)=f(x)+f^{\prime}(x)(t-x)+\frac{1}{2}f^{\prime\prime}(x)(t-x)^{2}
+r(t,x)(t-x)^{2},
\end{equation}
where $r(t,x)$ is the remainder term  and $\displaystyle\lim_{n\rightarrow\infty}r(t,x)=0.$
Operating $D_{n}^{\beta}$ to the equation (\ref{e18}), we obtain
\begin{align*}
D_{n}^{\beta}(f,x)-f(x) &= D_{n}^{\beta}(t-x,x)f^{\prime}(x)+ D_{n}^{\beta}\left(\left(t-x\right)^{2},x\right)
\frac{f^{\prime\prime}(x)}{2} \\
& \hspace{10mm} + D_{n}^{\beta}\left( r\left( t,x\right)  \left(  t-x\right)^{2},x\right)
\end{align*}
Using Cauchy-Schwarz inequality, we have
\begin{equation}\label{e19}
D_{n}^{\beta}\left(r\left(t,x\right)\left(t-x\right)^{2},x\right)\leq \sqrt{D_{n}^{\beta}\left(r^{2}\left(t,x\right)
,x\right)}\sqrt{D_{n}^{\beta}\left(\left(t-x\right)^{4},x\right)}.
\end{equation}
As $r^{2}\left(x,x\right)=0$ and $r^{2}\left(t,x\right) \in C_{2}^{\ast}[0,\infty)$, we have
\begin{equation}\label{e20}
\lim_{n\rightarrow\infty}D_{n}^{\beta}\left(r^{2}\left(t,x\right),x\right)=r^{2}\left(x,x\right) = 0
\end{equation}
uniformly with respect to $x\in\left[0,A\right].$ Now from (\ref{e19}),
(\ref{e20}) and from Remark \ref{r1}, we get
\begin{align*}
\lim_{n\rightarrow\infty} nD_{n}^{\beta}\left(r\left(t,x\right)\left(t-x\right)^{2},x\right)=0.
\end{align*}
Thus
\begin{align*}
\lim_{n\rightarrow\infty}n\left(D_{n}^{\beta}(f,x)-f(x)\right)
&= \lim_{n\rightarrow\infty}n\biggl[D_{n}^{\beta}(t-x,x)f^{\prime}(x)+\frac{1}{2}f^{\prime\prime
}(x)D_{n}^{\beta}(\left(t-x\right)^{2},x) \\
& \hspace{25mm} + D_{n}^{\beta}(r\left(t,x\right)\left(t-x\right)^{2},x)\biggr] \\
&= \frac{1}{1-\beta}f^\prime(x)+\frac{x}{1-\beta}f^{\prime\prime}(x).
\end{align*}
\end{proof}

By $C_B[0,\infty)$, we denote the class on real valued continuous bounded functions $f(x)$ for $x\in[0,\infty)$ with
the norm $||f||=\sup_{x\in[0,\infty)}|f(x)|$. For $f\in C_B[0,\infty)$ and $\delta >0$ the $m$-th order modulus
of continuity is defined as
$$\omega_m(f,\delta)=\sup_{0\le h\le \delta}\sup_{x\in[0,\infty)}|\Delta^m_hf(x)|,$$
where $\Delta$ is the forward difference. In case $m=1$ we mean the usual modulus of continuity denoted by
$\omega(f,\delta).$
The Peetre's $K$-functional is defined as
$$K_2(f,\delta)=\inf_{g\in C_B^2[0,\infty)}\left\{||f-g||+\delta||g^{\prime\prime}||:g\in C_B^2[0,\infty)\right\},$$
where $$C_B^2[0,\infty)=\{g\in C_B[0,\infty): g^\prime, g^{\prime\prime}\in C_B[0,\infty)\}.$$

\begin{theorem}\label{t3} Let $f\in C_B[0,\infty)$ and $\beta>0,$ then
\begin{align*}
|D_{n}^{\beta}(f,x)-f(x)| &\leq C \omega_2\left(f,\sqrt{\left[\frac{2x}{n(1-\beta)}+\frac{2}{n^2(1-\beta)}
+ \frac{1}{n^2(1-\beta)^2}\right]}\right) \\
& \hspace{35mm} + \omega\bigg(f,\frac{1}{n(1-\beta)}\bigg)
\end{align*}
where $C$ is a positive constant.
\end{theorem}

\begin{proof}
We introduce the auxiliary operators $\bar{D}_{n}^{\beta}:C_B[0,\infty)\to C_B[0,\infty)$ as follows
\begin{align}\label{e21}
\bar{D}_{n}^{\beta}(f,x) = D_{n}^{\beta}(f,x)-f\bigg(x+\frac{1}{n(1-\beta)}\bigg)+f(x),
\end{align}
These operators are linear and preserves the linear functions in view of Lemma \ref{l4}. Let $g\in C_B^2[0,\infty)$
and $x, \, t \in[0,\infty).$ By Taylor's expansion
\begin{align*}
g(t) = g(x)+(t-x)g^\prime(x)+\int_{x}^{t}(t-u)g^{\prime\prime}(u)du,
\end{align*}
we have
\begin{align}
|\bar{D}_{n}^{\beta}(g,x)-g(x)| &\leq \bar{D}_{n}^{\beta}\bigg(\bigg|\int_{x}^{t}(t-u)g^{\prime\prime}(u)du\bigg|,x\bigg)
\nonumber\\
&\leq D_{n}^{\beta}\bigg(\bigg|\int_{x}^{t}(t-u)g^{\prime\prime}(u)du\bigg|,x\bigg) \nonumber\\
& \hspace{20mm} + \bigg|\int_{x}^{x+\frac{1}{n(1-\beta)}}\bigg(x+\frac{1}{n(1-\beta)}-u\bigg)g^{\prime\prime}(u)du\bigg|
\nonumber\\
&\leq D_{n}^{\beta}((t-x)^2,x)\|g^{\prime\prime}\|+\bigg|\int_{x}^{x+\frac{1}{n(1-\beta)}}
\bigg(\frac{1}{n(1-\beta)}\bigg)du\bigg|\|g^{\prime\prime}\|\nonumber
\end{align}
Next, using Remark \ref{r1}, we have
\begin{align}
|\bar{D}_{n}^{\beta}(g,x)-g(x)| &\leq \bigg[ D_{n}^{\beta}((t-x)^2,x)+\bigg(\frac{1}{n(1-\beta)}\bigg)^2\bigg]\|
g^{\prime\prime}\|\nonumber\\
&\leq \bigg[ D_{n}^{\beta}((t-x)^2,x)+\bigg(\frac{1}{n(1-\beta)}\bigg)^2\bigg]\|g^{\prime\prime}\|\nonumber \\
&\leq \left[\frac{2x}{n(1-\beta)}+\frac{2}{n^2(1-\beta)}+\frac{1}{n^2(1-\beta)^2}\right]\|g^{\prime\prime}\|=\delta_n\|
g^{\prime\prime}\|.\label{e22}
\end{align}
Since
\begin{align*}
|D_{n}^{\beta}(f,x)|\le \sum_{k=0}^{\infty}\left(\int_0^\infty L_{n,k}^{(\beta)}(t)\,dt \right)^{-1}L_{n,k}^{(\beta)}(x)
\int_{0}^{\infty}L_{n,k}^{(\beta)}(t)|f(t)|\,dt\le\|f\|.
\end{align*}
Now by (\ref{e19}), we have
\begin{align}\label{e23}
|||\bar{D}_{n}^{\beta}(f,x)||\le |||D_{n}^{\beta}(f,x)||+2||f||\le 3||f||, \hspace{5mm} f \in C_B[0,\infty).
\end{align}
Using (\ref{e21}), (\ref{e22}) and (\ref{e23}), we have
\begin{align*}
|D_{n}^{\beta}(f,x)-f(x)| &\leq  |\bar{D}_{n}^{\beta}(f-g,x)-(f-g)(x)|+|\bar{D}_{n}^{\beta}(g,x)-g(x)|\\
& \hspace{25mm} + \bigg|f\bigg(x+\frac{1}{n(1-\beta)}\bigg)-f(x)\bigg|\\
&\leq 4\|f-g\|+\delta_n\|g^{\prime\prime}\|+\bigg|f\bigg(x+\frac{1}{n(1-\beta)}\bigg)-f(x)\bigg|\\
&\leq C\left\{\|f-g\|+\delta_n\|g^{\prime\prime}\|\right\}+\omega\left(f,\frac{1}{n(1-\beta)}\right).
\end{align*}
Taking infimum over all $g\in C_B^2[0,\infty)$, and using the inequality $K_2(f,\delta) \leq
C\omega_2(f,\sqrt{\delta})$, $\delta >0$ due to \cite{dl}, we get the desired assertion.
\end{proof}


\vspace{4mm}
\begin{center}
\bf{Acknowledgement}
\end{center}

The authors would like to thank Dr. Ali Aral for checking many of the results presented, valueable comments
and verification that \cite{af} contains minor errors. 



\end{document}